\documentclass[11pt]{amsart}

\usepackage{amssymb,comment}

\usepackage{subfigure, graphicx, tikz, color}
\usepackage{hyperref}     
\usepackage{graphicx} 
\usepackage[font=small,labelfont=bf]{caption}

\newcommand{\Ex}{\,{\mathbb{E}}}
\newcommand{\Prob}{{\mathbb{P}}}

\newcommand{\N}{{\mathbb N}}

\newcommand{\p}{{\mathbb P}}

\providecommand{\eps}{}
\renewcommand{\eps}{\varepsilon}

\newtheorem{thm}{Theorem}
\newtheorem{lem}[thm]{Lemma}

\theoremstyle{remark}

\newtheorem{remark}[thm]{Remark}

\usepackage[normalem]{ulem}

\definecolor{clou}{rgb}{0.8,0.25,0.5125}

\usepackage[numbers,sort&compress] {natbib}

\begin{document}


\title{Patricia's Bad Distributions}

\author[Addario-Berry]{Louigi Addario--Berry}
\address{Department of Mathematics and Statistics, McGill University, Montr\'eal, Canada}
\email{louigi.addario@mcgill.ca}

\author[Morin]{Pat Morin}
\address{School of Computer Science, Carleton University, Ottawa Ontario Canada}
\email{morin@scs.carleton.ca}

\author[Neininger]{Ralph Neininger}
\address{Institute of Mathematics, Goethe University Frankfurt, Frankfurt a.M., Germany}
\email{neininger@math.uni-frankfurt.de}

\begin{abstract}
The height of a random PATRICIA tree built from independent, identically distributed infinite binary strings with arbitrary diffuse probability distribution $\mu$ on $\{0,1\}^\N$ is studied. We show that the expected height grows asymptotically sublinearly in the number of leaves for any such $\mu$, but can be made to exceed any specific sublinear growth rate by choosing $\mu$ appropriately.
\end{abstract}

\date{May 18, 2024}

\maketitle



\section{Introduction and results}
The PATRICIA tree is a space efficient data structure for strings which improves on the trie. For the purpose of this note it is sufficient to introduce these tree structures for binary strings: Label the nodes of the complete infinite rooted binary tree by the elements of $\cup_{k=0}^\infty\{0,1\}^k$, starting at the root with $\emptyset$ and left and right child of a node labelled $v\in \{0,1\}^k$ with $v0$ and $v1$, respectively. Here, for $v\in\{0,1\}^k$ with $v=(v_1,\ldots,v_k)$ we abbreviate $v$ as $v=v_1\ldots v_k$ and denote $vi:=v_1v_2\ldots v_ki$ for $i=0,1$.

The coming definitions are depicted in Figure~\ref{fig_1}.
For distinct infinite binary strings $x_1,\ldots,x_n\in\{0,1\}^\N$ a finite tree called a {\em trie} (or {\em radix search tree}) to represent the strings $x_1,\ldots,x_n$ is constructed by first associating with each $x_i$ the infinite path in $\cup_{k=0}^\infty\{0,1\}^k$ consisting of the nodes whose labels are the prefixes of $x_i$. The node labelled with the shortest such prefix that is not a prefix of any $x_j$ with $j\in\{1,\ldots,n\}\setminus\{i\}$ becomes a leaf in the trie representing string $x_i$ for $i=1,\ldots,n$. The resulting tree, which is a finite binary tree with $n$ leaves, is the trie representing $x_1,\ldots,x_n$. Next, starting from the trie, all vertices with out-degree $1$ (i.e.\ with exactly one child) are deleted and the resulting gaps are closed by merging the two nodes which formed a deleted edge. This results in the PATRICIA tree, which was introduced independently by Morrison \cite{mo68} and Gwehenberger \cite{gw68} and first systematically analysed by Knuth \cite{kn73}.  The PATRICIA tree contains all the information needed to retrieve the strings and to perform  operations such as sorting, searching and selecting; for broad expositions, see \cite{kn98, ma92, sz01}.

PATRICIA trees have been analysed assuming various probabilistic models for the input strings; where usually the infinite strings are assumed to be independent and identically distributed over $\{0,1\}^\N$. Note that atoms of such a distribution result in identical strings with positive probability, and in this case the construction of the trie does not lead to a finite tree. Hence, the law of the strings is usually assumed to be diffuse (non-atomic). Special cases of such diffuse probability distributions have been considered in the analysis of algorithms on strings such as the Bernoulli models, Markov model, dynamical sources or the density model; see \cite{pi85,pi86,de92,clflva01,ja12,fuhwza13,lenesz13,ak22,ja22,isch23} and the references given in these papers.

\begin{center}\label{fig_1}
\includegraphics[width=11cm]{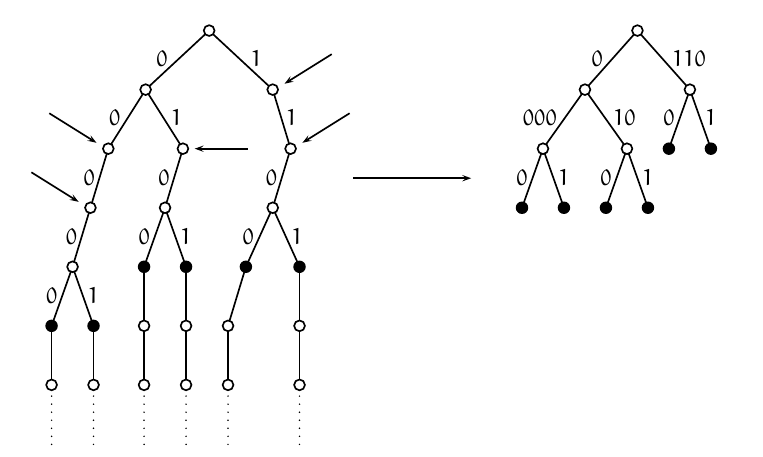}
\captionof{figure}{On the left the trie for the strings $00000\ldots$, $00001\ldots$, $0100\ldots$, $0101\ldots$, $1100\ldots$, and $1101\ldots$ is shown. Its leaves are the full black vertices, the indicated children of the full black vertices do not belong to the trie. Vertices with out-degree $1$ within the trie are indicated by arrows. On the right the resulting PATRICIA tree by deleting corresponding edges is shown.}
\end{center}

In the present note we focus on the height of a PATRICIA tree, which is the maximal (graph) distance of any leave from the root. The asymptotic behavior of the height of tries and PATRICIA trees under the Bernoulli models is covered by Pittel \cite{pi85,pi86} and
Devroye \cite{de92b,de98}. For example, for the height $H^\mathrm{syB}_n$ of the PATRICIA tree constructed from $n$ independent strings under the symmetric Bernoulli model, i.e.~all bits being independent and Bernoulli$(\frac{1}{2})$ distributed, Pittel \cite{pi85} obtained as $n\to\infty$ that
\begin{align*}
\frac{H^\mathrm{syB}_n}{\log n} \to 1  \mbox{ almost surely.}
\end{align*}
This shows an asymptotic 50\% improvement of the PATRICIA tree over the trie, for which the limit constant for the same probabilistic model is $2$ instead of $1$. For general diffuse laws concentration of the height of PATRICIA trees is studied (assuming only independence of the infinite strings not necessarily identical distribution) by Devroye \cite{de05} based on results from \cite{boluma00}; see also \cite{knsz02} for concentration of the height of PATRICIA trees in the Bernoulli model.

While such studies aim to show that the height behaves well with respect to applications from algorithms, Evans and Wakolbinger \cite{evwa17,evwa20} studied these random tree structures as  tree-valued transient Markov chains from the perspective of Doob--Martin boundary theory. They asked (private communication) how high PATRICIA trees can grow for arbitrary diffuse probability distributions of the strings (see \cite[Section 5]{evwa20} for specific examples). The subject of the present note is to answer this question by Theorems \ref{thm_1} and \ref{thm_2}:  The expected height grows always sublinearly, but can be made to exceed any fixed sublinear growth rate by the choice of an appropriate diffuse law.

For a diffuse probability distribution $\mu$ on $\{0,1\}^\N$ and $(\Xi^{(j)})_{j \in \N}$ a sequence of independent and identically distributed random strings with law $\mu$ we denote by $H^\mu_n$ the height of the PATRICIA tree constructed from $\Xi^{(1)},\ldots,\Xi^{(n)}$.
\begin{thm}\label{thm_1}
For all diffuse probability distributions $\mu$ on $\{0,1\}^\N$ we have, as $n\to\infty$, that
\begin{align*}
\frac{\Ex[H^\mu_n]}{n}\to 0, \quad\mbox{ and } \quad \frac{H^\mu_n}{n} \to 0 \mbox{ almost surely.}
\end{align*}
\end{thm}
\begin{thm}\label{thm_2}
For any sequence $\alpha=(\alpha_n)_{n\in \N}$ of positive numbers with  $\alpha_n \to \infty$ as $n \to \infty$ there exists a diffuse probability distribution $\nu=\nu^{(\alpha)}$ on $\{0,1\}^\N$ such that for all $n$ sufficiently large
\begin{align*}
\frac{\Ex[H^\nu_n]}{n/\alpha_n} \to \infty, \quad\mbox{ and } \quad \frac{H^\nu_n}{n/\alpha_n} \to \infty \mbox{ almost surely.}
\end{align*}
\end{thm}

 We call a law $\nu$ on $\{0,1\}^\N$ causing large expected heights $\Ex[H^\nu_n]$ {\em bad} since such laws are undesirable from the point of view of the efficiency of algorithms based on PATRICIA trees. The remaining part of the present note contains proofs of these two theorems.
\begin{remark}
 For the density model, which is a subclass of the diffuse distributions on $\{0,1\}^\N$, the asymptotics of Theorem \ref{thm_1} were obtained by Devroye \cite[page 419]{de92}. There, also bad distributions with asymptotic properties as in our Theorem \ref {thm_2}  are constructed for sequences $\alpha_n=n^\varepsilon$ with $0<\varepsilon<1$.
 \end{remark}

\section{Proofs}

\subsection{Proof of Theorem \ref{thm_1}}
We start with a technical observation:
\begin{lem}\label{lem_level_k}
Suppose $\mu$ is a diffuse probability distribution on $\{0,1\}^\N$, and let $\Xi=(\xi_i)_{i\in\N}$ be random with law $\mu$.
Then for all $\eps$ there exists $k = k(\eps) \in \N$ such that for any string $v=v_1\ldots v_k \in \{0,1\}^k$, $\p(\xi_1\ldots\xi_k=v_1\ldots v_k) <\eps$.
\end{lem}
\begin{proof}
Suppose for a contradiction that there exists $\eps > 0$ such that for all $k \in \N$ there is a string $v_1\ldots v_k \in \{0,1\}^k$ such that $\p(\xi_1\ldots\xi_k=v_1\ldots v_k) \ge \eps$. Then by a compactness argument shown below there exists an infinite string $v=(v_i)_{i\in\N} \in \{0,1\}^{\N}$ such that for all $k \in \N$, $\p(\xi_1\ldots\xi_k=v_1\ldots v_k) \ge \eps$. The events $\{\xi_1\ldots\xi_k=v_1\ldots v_k\}$ are decreasing in $k$, so this implies that
\[
\p(\Xi=v) = \lim_{k \to \infty}\p(\xi_1\ldots\xi_k=v_1\ldots v_k) \ge \eps\, ,
\]
which contradicts the assumption that $\mu$ is diffuse.

It remains to show the existence of the infinite string $v=(v_i)_{i\in\N} \in \{0,1\}^\N$ such that for all $k \in \N$, $\p(\xi_1\ldots\xi_k=v_1\ldots v_k) \ge \eps$. Consider $\{0,1\}$ as a topological space with the discrete topology (all subsets being open) and $\{0,1\}^\N$ as the product space with the product topology. As a product of compact spaces $\{0,1\}^\N$ is compact. It is also a Hausdorff space. The projections $\Pi_k:\{0,1\}^\N \to \{0,1\}^k$ given by
\[(v_i)_{i\in\N} \stackrel{\Pi_k}{\longmapsto} v_1\ldots v_k
\]
are continuous for all $k\in \N$. Hence, the set
\begin{align*}
V_k&:=\{(v_i)_{i\in\N}\in \{0,1\}^\N\,|\, \p(\xi_1\ldots\xi_k=v_1\ldots v_k) \ge \eps\}\\
&= \bigcup_{v_1\ldots v_k \in \{0,1\}^k \atop \p(\xi_1\ldots\xi_k=v_1\ldots v_k) \ge \eps} \Pi_k^{-1}(\{v_1\ldots v_k\})
\end{align*}
is closed and thus compact in $\{0,1\}^\N$. This implies that $(V_k)_{k\in \N}$ is a nested sequence of non-empty, compact sets. Now, Cantor's intersection theorem implies
\begin{align*}
\bigcap_{k=1}^\infty V_k \neq \emptyset.
\end{align*}
Any element $v$ of $\bigcap_{k=1}^\infty V_k$ has the desired property.
\end{proof}

\begin{proof}[Proof of Theorem \ref{thm_1}.] Fix a diffuse probability distribution $\mu$ on $\{0,1\}^\N$. Let $\Xi^{(j)}=(\xi^{(j)}_i)_{i\in\N}$ for $j\in\N$ be independent, identically distributed with law $\mu$ and denote by $T_n$ the PATRICIA tree built from $\Xi^{(1)},\ldots,\Xi^{(n)}$.

We first show that $H^\mu_n/n\to 0$ almost surely. Fix any $\eps \in (0,1/4)$.
Let $k=k(\eps)$ be as in Lemma \ref{lem_level_k}, so that for any string $v=v_1\ldots v_k \in \{0,1\}^k$, if $\Xi=(\xi_i)_{i \in \N}$ has law $\mu$ then $\p(\xi_1\ldots\xi_k=v_1\ldots v_k) <\eps$.
To prove $H^\mu_n/n\to 0$ almost surely we first show that
\begin{align}\label{as_conv}
\p(\exists\, n_0\, \forall\, n \ge n_0: H^\mu_n \le k+2\eps n)=1.
\end{align}
Note that if the event
\[
E_{n,k}:= \bigcup_{v_1\ldots v_k\in \{0,1\}^k} \{|\{1 \le j \le n: \xi^{(j)}_1\ldots\xi^{(j)}_k=v_1\ldots v_k\}| \ge 2\eps n\}
\]
does not occur then the subtrees of $T_n$ rooted at nodes $v \in \{0,1\}^k$ all have at most $2\eps n$ leaves and so height less than $2 \eps n$; thus if $E_{n,k}$ does not occur then $H^\mu_n \le k+2\eps n$. It follows that
\begin{align*}
\lefteqn{
\p(\exists\, n_0\, \forall\, n \ge n_0: H^\mu_n \le k+2\eps n)}\\
&\ge \p(E_{n,k}\mbox{ occurs for at most  finitely many values }n)\\
&= \p\left(\left(\limsup_{n\to\infty}E_{n,k}\right)^c\right)
\, ,
\end{align*}
so to prove (\ref{as_conv}) it suffices to show that the  probability of $\limsup_{n\to\infty}E_{n,k}$ is $0$. For this, simply note that
\begin{align*}
\p(E_{n,k})
& \le \sum_{v_1\ldots v_k \in \{0,1\}^k}
\p(|\{1 \le j \le n: \xi^{(j)}_1\ldots\xi^{(j)}_k=v_1\ldots v_k\}| \ge 2\eps n) \\
& \le 2^k \p(Y_n\ge 2\eps n)\, ,
\end{align*}
where $Y_n$ has the Binomial distribution Bin$(n,\eps)$;
the second inequality holds since the events that $\xi^{(j)}_1\ldots\xi^{(j)}_k=v_1\ldots v_k$ are independent for distinct $1 \le j \le n$, and each has probability at most $\eps$.
A Chernoff bound then gives
\[
\p(E_{n,k}) \le 2^k e^{-\eps n/2}.
\]
Since this is summable, it follows by the first Borel--Cantelli lemma that
\[
\p\left(\limsup_{n\to\infty}E_{n,k}\right) = 0,
\]
hence we obtain (\ref{as_conv}). Now, note that for any $m_0\in\N$,
\begin{align*}
\left\{\frac{H^\mu_n}{n}\to 0\right\}=\bigcap_{m=m_0}^\infty \bigcup_{n_0=1}^\infty \bigcap_{n=n_0}^\infty \left\{\frac{H^\mu_n}{n}\le \frac{3}{m}\right\}\, .
\end{align*}
 Thus, for $\eps = \frac{1}{m}$ with fixed $m\ge m_0$ we can choose $n$ sufficiently large so that $k(\eps)/n\le \eps$ and obtain
\begin{align*}
\left\{H^\mu_n\le k(\eps)+2\eps n\right\}\subset \left\{\frac{H^\mu_n}{n}\le \frac{3}{m}\right\}
\end{align*}
and see that (\ref{as_conv}) implies $H^\mu_n/n\to 0$ almost surely.

Finally, note that by construction of the PATRICIA tree we deterministically have $H^\mu_n\le n-1$, thus $H^\mu_n/n\le 1$. Hence, we obtain from $H^\mu_n/n\to 0$ almost surely and dominated convergence that $\Ex[H^\mu_n]/n\to 0$.
\end{proof}

\subsection{Proof of Theorem \ref{thm_2}}
As building blocks for our bad distributions we first define a set of auxiliary probability distributions $(\mu_N$, $N\in\N)$, on $\{0,1\}^\N$ as follows. For fixed $N\in\N$ we choose $T$ uniformly at random from $\{1,\ldots,N^2\}$. Independently of $T$, let $(B_i)_{i\in \N}$ be independent  Bernoulli$(\frac{1}{2})$-distributed random variables. Then define a sequence $(\vartheta_i)_{i \in \N}$ by
\begin{align}\label{def_xi}
 \vartheta_i=
 \begin{cases}
 0, & \mbox{if } i< T,   \\
 1, & \mbox{if } i= T, \\
 B_{i-T},& \mbox{if } i> T.
 \end{cases}
\end{align}
Now, $\mu_N$ is defined as the law of the string $\Theta=(\vartheta_i)_{i\in\N}$. Note that by definition $\mu_N$ is diffuse for all $N\in \N$. We use the notation
\begin{align*}
\langle\Theta\rangle:=\min\{i\in\N\,|\,\vartheta_i=1\}
\end{align*}
for the index of the first entry of $\Theta$ equal to $1$.
\begin{lem}\label{upper_bound} For any $n\in\{1,\ldots,N\}$ we have $\Ex[H^{\mu_N}_n] \ge n-2$.
\end{lem}
\begin{proof}
Let $1\le n\le N\in\N$ and $\Theta^{(1)},\ldots,\Theta^{(n)}$ be i.i.d.~with law $\mu_N$.
We consider the set $A:=\{ \langle\Theta^{(1)}\rangle,\ldots,\langle\Theta^{(n)}\rangle\}\subset\{1,\ldots,N^2\}$. By construction of the PATRICIA tree we have
\begin{align}\label{height_1}
H^{\mu_N}_n \ge |A|-1,
\end{align}
where $|A|$ denotes the cardinality of $A$, i.e., the number of distinct elements within the set $\{\langle\Theta^{(1)}\rangle,\ldots,\langle\Theta^{(n)}\rangle\}$. For all $1\le i<j\le n$ we have
$\Prob(\langle\Theta^{(i)}\rangle=\langle\Theta^{(j)}\rangle)=1/N^2$. Hence, we obtain
\begin{align}\label{height_2}
\Ex[|A|]\ge n - \Ex\left[\sum_{1\le i<j\le n}\mathbf{1}_{\{\langle\Theta^{(i)}\rangle=\langle\Theta^{(j)}\rangle\}}\right]\ge n- \frac{n^2}{2N^2}\ge n-1,
\end{align}
since $n\le N$. Now, \eqref{height_1} and \eqref{height_2} imply the assertion.
\end{proof}

\begin{proof}[Proof of Theorem \ref{thm_2}.]
Without loss of generality we may assume that $\alpha_n=o(n)$. There exists an $n_0\in\N$ such that $\alpha_n\ge 8$ for all $n\ge n_0$. We define  $\beta_n:=\lfloor \log_2 \alpha_n\rfloor -2$ and a sequence $(A(n))_{n\in \N}$ as a generalized inverse of $(\beta_n)_{n\in \N}$ by
\begin{align}\label{def:Ak}
A(n):=\max\{m\in\N\,|\, \beta_m\le n\},\quad n\in\N.
\end{align}
First, a probability distribution $\mu^{(\alpha)}$ on $\{0,1\}^\N$ is obtained in two stages. Let $G$ be a random variable with geometric distribution with parameter $\frac{1}{2}$, i.e., with $\Prob(G=k)=(\frac{1}{2})^k$ for $k\in\N$.  Then define a sequence $(\lambda_i)_{i \in \N}$ by
\begin{align}\label{def_phi}
\lambda_i=\left\{\begin{array}{cl}
0, & \mbox{if } i< G, \\
1, & \mbox{if } i= G, \\
\vartheta_{i-G},& \mbox{if } i> G,
\end{array}
 \right.
\end{align}
where $\Theta=(\vartheta_i)_{i\in\N}$, conditional on $\{G=k\}$, has law $\mu_{A(k)}$  defined in \eqref{def_xi} with $A(\cdot)$ defined in \eqref{def:Ak}. We then define $\mu=\mu^{(\alpha)}$ as the law of $\Lambda=(\lambda_i)_{i\in \N}$.
Since the $\mu_{A(k)}$ are diffuse, we obtain that $\mu$ is diffuse.

Now, let $\Lambda^{(j)}=(\lambda^{(j)}_i)_{i\in\N}$ for $j \in \N$ be independent with law $\mu$. For $n\ge n_0$, by construction,
\begin{align*}
X_n := \left|\left\{1\le j\le n \,:\, \left(\lambda^{(j)}_1,\ldots,\lambda^{(j)}_{\beta_n}\right)=(0,\ldots,0,1)\right\}\right|
\end{align*}
is $\mathrm{Bin}(n,2^{-\beta_n})$-distributed. To get rid of the floors in the definition of $\beta_n$ denote by $X'_n$ a $\mathrm{Bin}(n,4/\alpha_n)$-distributed random variable. Note that $2^{-\beta_n}\ge 4/\alpha_n$. By Okamoto's inequality, see \cite{ok58} or \cite[Exercise 2.12]{boluma13}  we have
\begin{align*}
\Prob\left(X_n< \frac{2n}{\alpha_n}\right)&\le \Prob\left(X'_n-\frac{4n}{\alpha_n}< -\frac{2n}{\alpha_n}\right)\\
&\le \exp\left(-\frac{n(2/\alpha_n)^2}{2(4/\alpha_n)(1-4/\alpha_n)}\right) \le\exp\left(-\frac{n}{2\alpha_n}\right).
\end{align*}
Hence, with high probability at least $\lceil 2n/\alpha_n\rceil$ of the $n$ strings start with the prefix $(0,\ldots,0,1)$ of length $\beta_n$ and thus have suffixes $(\lambda^{(j)}_{\beta_n+1},\lambda^{(j)}_{\beta_n+2},\ldots)$ drawn independently from $\mu_{A(\beta_n)}$ for the respective $j$. For all $n\ge n_0$ we have  $\lceil 2n/\alpha_n\rceil\le n \le A(\beta_n)$. Hence, by Lemma \ref{upper_bound}, $\lceil 2n/\alpha_n\rceil$ such strings cause an expected height of at least $2n/\alpha_n-2$. Together we obtain for all sufficiently large $n$, note also $\alpha_n=o(n)$, that
\begin{align*}
\Ex[H^{\mu}_n]&\ge \Prob\left(X_n\ge \frac{2n}{\alpha_n}\right)\Ex\left[H^{\mu}_n\Big|X_n\ge \frac{2n}{\alpha_n}\right]\\
&\ge
\left(1-\exp\left(-\frac{n}{2\alpha_n}\right)\right)\left(\frac{2n}{\alpha_n}-2\right)\\
&\ge \frac{n}{\alpha_n}.
\end{align*}
Since the sequence $(\log \alpha_n)$ tends to infinity the present proof implies the existence of a diffuse  probability distribution $\nu=\nu^{(\alpha)}$ on $\{0,1\}^\N$ such that $\Ex[H^{\nu}_n]\ge n/\log \alpha_n$ for all sufficiently large $n\in\N$, hence
\begin{align*}
\frac{\Ex[H^\nu_n]}{n/\alpha_n}\to \infty.
\end{align*}

To prove the second statement of Theorem \ref{thm_2} we use the following bound from Devroye \cite[page 21]{de05}: for any diffuse probability distribution $\mu$ on $\{0,1\}^\N$ and any $t>0$,
\begin{align}\label{dev_bound}
\Prob\left(H^\mu_n \le \Ex[H^\mu_n] -t\right)\le \exp\left(-\frac{t^2}{2\Ex[H^\mu_n+1]}\right)\le \exp\left(-\frac{t^2}{2n}\right).
\end{align}
We now consider the probabilities
\begin{align}\label{dev_bound_b}
\Prob\left(H^{\nu}_n \le \frac{n}{\log^2 \alpha_n}\right)= \Prob\left(H^\nu_n \le \Ex[H^{\nu}_n]-\left(\Ex[H^{\nu}_n]-\frac{n}{\log^2 \alpha_n} \right)\right)
\end{align}
and note that for all sufficiently large $n$ we have
\begin{align} \label{dev_bound_c}
\Ex[H^{\nu}_n]-\frac{n}{\log^2 \alpha_n}\ge \frac{n}{\log \alpha_n}-\frac{n}{\log^2 \alpha_n}=
n\,\frac{\log(\alpha_n) -1}{\log^2 \alpha_n}\ge \frac{n}{\log^2 n}.
\end{align}
Combining (\ref{dev_bound})--(\ref{dev_bound_c}) we obtain
\begin{align*}
\Prob\left(H^{\nu}_n\le \frac{n}{\log^2 \alpha_n}\right)&\le \exp\left(-\frac{n}{2\log^4 n}\right)
\end{align*}
for all sufficiently large $n$. Since these upper bounds are summable it follows from the  first  Borel--Cantelli Lemma that $\liminf_{n\to\infty} H^{\nu}_n/(n/\log^2 \alpha_n) \ge 1$ almost surely, hence
\begin{align*}
\frac{H_n^{\nu}}{n/\alpha_n} \to  \infty \mbox{ almost surely.}
\end{align*}
Thus, $\nu$ has the properties claimed in Theorem \ref{thm_2}.
 \end{proof}

\subsubsection*{Acknowledgement}
This research was mainly done during the Sixteenth Annual Workshop on Probability and Combinatorics at McGill University's Bellairs Institute in Holetown, Barbados. We thank Bellairs Institute for its hospitality and support. We also thank the referees and Jasmin Straub for comments on a draft of this note.

\bibliographystyle{amsplain}
\providecommand{\bysame}{\leavevmode\hbox to3em{\hrulefill}\thinspace}
\providecommand{\MR}{\relax\ifhmode\unskip\space\fi MR }
\providecommand{\MRhref}[2]{%
  \href{http://www.ams.org/mathscinet-getitem?mr=#1}{#2}
}
\providecommand{\href}[2]{#2}

\end{document}